\documentclass[12pt]{amsart}
\newtheorem{theorem}{Theorem}
\newtheorem{proposition}[theorem]{Proposition}

\begin{document}

\baselineskip=17pt

\title[Unipotent subgroups of real semisimple Lie groups]{On the conjugacy 
of maximal unipotent subgroups of real semisimple Lie groups}

\author[H. Azad]{Hassan Azad}

\address{Department of Mathematics and Statistics,
King Fahd University of Petroleum and Minerals,
Dhahran 31261, Saudi Arabia}

\email{hassanaz@kfupm.edu.sa}

\author[I. Biswas]{Indranil Biswas}

\address{School of Mathematics, Tata Institute of Fundamental
Research, Homi Bhabha Road, Bombay 400005, India}

\email{indranil@math.tifr.res.in}

\date{}


\begin{abstract}
The existence of closed orbits of real algebraic groups on real
algebraic varieties is established. As an application, it is shown
that if $G$ is a real reductive linear group with
Iwasawa decomposition $G\,=\, KAN$, then every unipotent
subgroup of $G$ is conjugate to a subgroup of $N$.
\end{abstract}

\maketitle

The principal aim of this note is to give a new proof of the following result.

\begin{proposition}\label{prop1}
Let $G$ be a connected real reductive linear group, and let
$G\, =\, KAN$ be an Iwasawa decomposition. Then any unipotent subgroup
$U$ of $G$ is conjugate to a subgroup of $N$.
In particular, all maximal unipotent subgroups are conjugate in $G$.
\end{proposition}

The novelty of the proof is that it is entirely geometric and relatively elementary.

Proposition \ref{prop1} is a special case of a result in Borel-Tits \cite[p.
126]{BT}. It is also proved in Mostow \cite[Theorem 2.1]{Mo}, and
in Onishchik-Vinberg \cite[p. 276, Theorem 7]{OV} and Vinberg \cite{Vi}. The
proposition seems not be as well known as it should be. For example, it is not
given in standard references like \cite{He}, or \cite{Kn}, although the one
dimensional case is proved in \cite[p. 431]{He} using the
Jacobson-Morozov theorem. (The reference in \cite{BT} was pointed out by T. N.
Venkataramana after a first draft of the note had been written.) 

The proof in \cite{Mo}, which is explained in \cite[\S~2.1--2.7]{Mo} is done by 
induction on the dimension of $G$, using properties of chambers associated to 
split diagonalizable subalgebras of the Lie algebra of $G$. The key step in the 
proof in \cite{OV} is that a connected triangular
subgroup of $\text{GL}(V)$, where $V$ is a finite dimensional real vector
space, has a fixed point in any invariant closed subset of the flag variety for
$V$ (this is deduced using Lemma 3 in \cite[p. 276]{OV}).

A knowledge of solvable subgroups is of importance in theoretical physics,
as explained in the papers of Patera, Winternitz and Zassenhaus
\cite{PWZ1}, \cite{PWZ2}, where
the authors have determined all maximal solvable connected subgroups of the
classical real groups. The classification of solvable subgroups is also of
great practical use in the reduction theory of differential equations.

Proposition \ref{prop1} can be proved using the ideas in the proof of Theorem
3.1 of \cite{AL}. However, we give a proof using known elementary results from 
real algebraic geometry instead of recreating them ab initio.

Let $\mathbb L$ be a connected affine complex algebraic group defined over
$\mathbb R$.
Let $L\,\subset\, {\mathbb L}({\mathbb R})$ be the connected component,
containing the identity element, of the locus of real points ${\mathbb 
L}({\mathbb R})$. Let $\mathbb X$ be an irreducible complex algebraic
variety defined over $\mathbb R$.
Let $X$ be a connected component of ${\mathbb X}(\mathbb R)$, the set of real 
points of $\mathbb X$.

\begin{proposition}\label{lem1}
Let $\phi\, :\, L\times X \,\longrightarrow\, X$ be an action of $L$ on $X$
satisfying the following condition: there is an algebraic action
${\mathbb L}\times{\mathbb X}\,\longrightarrow\, {\mathbb X}$ defined
over $\mathbb R$ that induces $\phi$. Then $L$ has a closed orbit in $X$.
\end{proposition}

\begin{proof}
If $\mathbb L$ is a complex affine algebraic group acting algebraically on a 
complex algebraic variety $\mathbb X$, then it is known that $\mathbb L$ 
has a closed orbit in $\mathbb X$. Indeed, this is an immediate consequence of 
the fact that the boundary of any orbit consists of orbits of strictly smaller 
dimension (cf. \cite[p. 19]{St}, \cite[p. 53, Proposition 1.8]{Bo}). As shown 
below, the proposition can be derived from this fact.

Consider the dimensions of all the $L$ orbits in $X$.
Let $d$ be the smallest among these. Take a point $x_0\, \in\, X$
such that the orbit $L\cdot x_0$ is of dimension $d$. We will show that
this orbit $L\cdot x_0\, \subset\, X$ is closed.

Consider the action ${\mathbb L}\times{\mathbb X}\,\longrightarrow\,{\mathbb 
X}$ defined over $\mathbb R$ inducing $\phi$. For an irreducible smooth complex 
algebraic variety $\mathbb Y$, of dimension $d$, defined over $\mathbb R$, each
connected component of the set of real points of $\mathbb Y$ is a real manifold 
of dimension $d$ \cite[p. 8, (1.14)]{Si}. Hence the complex dimension of the orbit 
${\mathbb L} \cdot x_0$, which is a smooth irreducible 
algebraic variety defined over $\mathbb R$, is $d$. Therefore, the 
boundary of ${\mathbb L}\cdot x_0$ consists of orbits of dimensions strictly 
smaller than $d$.
Consequently, if the orbit $L\cdot x_0$ is not closed, for any point
$y\, \in\, \overline{L\cdot x_0} \setminus L\cdot x_0$ in the boundary, the 
orbit $L\cdot y$ will have dimension strictly smaller than $d$. But the
dimension of any $L$ orbit in $X$ is at least $d$. Hence we conclude that
the orbit $L\cdot x_0$ is closed.
\end{proof}

A real algebraic group $G$ is, for the purposes of this paper, a closed
connected subgroup of ${\rm GL}(n,{\mathbb R})$ such that the connected
Lie subgroup
$G^{\mathbb C}$ of ${\rm GL}(n,{\mathbb C})$ whose Lie algebra is
$\text{Lie}(G)+\sqrt{-1}\cdot\text{Lie}(G)$ is an affine algebraic
subgroup of ${\rm GL}(n,{\mathbb C})$, where $\text{Lie}(G)$ is the
Lie algebra of $G$. Every unipotent subgroup of ${\rm GL}(n,{\mathbb R})$
(meaning, a connected subgroup consisting of unipotent matrices) is algebraic.

\medskip
\noindent
\textbf{Proof of Proposition \ref{prop1}.}\, A real semisimple linear group is
real algebraic,
because its complexification is generated by unipotent subgroups
(see \cite[Ch.~3, \S~9.8]{Bou}).
Therefore, the group $G$ in Proposition \ref{prop1} is real algebraic.
Let $G\,=\, KAN$ be the Iwasawa decomposition of $G$. Let $H$ be the
Zariski closure of $AN$ in $G$; so $H$ is a finite extension of $AN$.

As noted above, the unipotent subgroup $U$ of $G$ is algebraic.
The approach we have taken is that a connected subgroup $G$ of ${\rm 
GL}(n,{\mathbb R})$ is 
real algebraic if the connected subgroup of ${\rm GL}(n,{\mathbb C})$ with Lie 
algebra $\text{Lie}(G)\oplus \sqrt{-1}\cdot \text{Lie}(G)$ is affine algebraic.
It is not completely obvious --- but standard --- that the complexification of
$A$ is an algebraic torus. The reason is that $A$ and a maximal torus of the
centralizer of $A$ in $K$ give a Cartan subgroup of $G$, hence its
complexification is a maximal torus, so if we restrict the Cartan
involution followed by taking inverse to this Cartan subgroup, then the
fixed point set has $A$ as a connected component.

Consider the 
action of $U$ on $G/H$. By Proposition \ref{lem1}, there is an element $\xi_0\, 
\in\, G$ such that the orbit $U\xi_{0}H/H$ in $G/H$ is closed. Since $G/H$ is 
compact, we conclude that the closed orbit $U\xi_{0}H/H$ is compact.

On the other hand, being an orbit of a unipotent group, $U\xi_{0}H/H\, \subset
\, G/H$ must be a cell. Therefore, $U\xi_{0}H/H$ is a point. Hence
$$
\xi^{-1}_{0} U\xi_0\, \subset\, H\, .
$$
Since the only unipotent elements of $H$ are in $N$, this implies that
$\xi^{-1}_{0} U\xi_0\, \subset\, N$. This completes the proof of
Proposition \ref{prop1}.$\hfill{\Box}$
\medskip

We thank Ernest Vinberg for suggesting that the proof of the
following result be also given along the lines of the preceding proof.

\begin{proposition}{\rm [Vinberg-Mostow, \cite{Vi}, \cite{Mo}].}\label{prop3}
Let $G\,=\, KAN$ be as in Proposition \ref{prop1} with
$G\, \subset\, {\rm GL}(V)$, where $V$ is a finite dimensional real vector
space. Let $S$ be a connected solvable subgroup of $G$ with all real
eigenvalues. Then $S$ is conjugate to a subgroup of $AN$.
\end{proposition}

\begin{proof}
By considering the Zariski closure of $S$ in $\text{GL}(V\otimes{\mathbb C})$
and taking the connected component, containing the identity element, of the
group of its real points, we may assume that $S$ is a real algebraic solvable 
group. By \cite[p. 449]{HN}, the group $S$ is then topologically a cell. 
Moreover, algebraic subgroups of $S$ are all connected because all the 
eigenvalues of all elements of $S$ are positive. Therefore, by \cite[p. 449]{HN}, 
the orbits of $S$ under an algebraic action are cells. One now argues as in the 
proof of Proposition \ref{prop1} to conclude that the group $S$ is conjugate to a
subgroup of $AN$.
\end{proof}

To determine maximal connected solvable algebraic groups of $G$, one notices
that the unipotent part is normalized by the semisimple elements in the
solvable group. Specific detailed information is in the basic papers of
Patera-Winternitz-Zassenhaus \cite{PWZ1}, \cite{PWZ2} and in
Snobel-Winternitz \cite{SW}.

\medskip
\noindent
\textbf{Acknowledgements.}\, The authors are grateful to Professor Ernest 
Vinberg and Professor Karl-Hermann Neeb for their critical comments and 
suggestions that led to substantial improvements of the manuscript. The 
first-named author thanks KFUPM for funding Research Project IN101026.

\end{document}